\documentclass[11pt]{article}
\usepackage{epsfig}
\usepackage{amssymb,amsmath,amsthm,amscd}
\usepackage{latexsym}
\usepackage{fixltx2e}
\usepackage{mathrsfs}
\usepackage{verbatim}

\newtheorem{thm}{Theorem}[section]
\newtheorem{lem}[thm]{Lemma}
\newtheorem{cor}[thm]{Corollary}
\theoremstyle{definition}
\newtheorem{defn}[thm]{Definition}

\newcounter{labelflag} 

\newcommand{\be}{\begin{equation}}
\newcommand{\ee}{\end{equation}}

\newcommand{\R}{\mathbb{R}}
\newcommand{\N}{\mathbb{N}}

\def \cali {{  {\mathcal{I}} }}
\def \calm {{  {\mathcal{M}} }}
\def \cals {{  {\mathcal{S}} }}
\def \calf {{  {\mathcal{F}} }}
 
\def \cala {{  {\mathcal{A}}  }}

\def \calb {{  {\mathcal{B}}  }}

\begin{document}

\begin{titlepage}
\title{\Large\bf   
 Periodic and Almost Periodic 
 Random Inertial Manifolds
 for Non-Autonomous  Stochastic 
 Equations
   }
\vspace{7mm}

\author{ 
 Bixiang Wang  
\vspace{1mm}\\
Department of Mathematics, New Mexico Institute of Mining and
Technology \vspace{1mm}\\ Socorro,  NM~87801, USA \vspace{5mm}\\
 Email:    bwang@nmt.edu\\
}
 
\date{}
\end{titlepage}

\maketitle

\begin{abstract}
 By the Lyapunov-Perron method,
 we prove the existence of random inertial manifolds
 for a class of equations driven simultaneously by
 non-autonomous deterministic and stochastic forcing.
 These  invariant manifolds contain  tempered
 pullback random attractors if such attractors exist.
 We also prove pathwise periodicity and almost
 periodicity of inertial manifolds when 
 non-autonomous deterministic forcing
 is periodic and almost periodic in time, respectively.
\end{abstract}



\baselineskip=1.2\baselineskip

\section{Introduction}\label{intr}
\setcounter{equation}{0}

In this paper, we investigate the existence of random
inertial manifold (IM) for the following non-autonomous stochastic
equation on a  separable
 Hilbert space $H$ for $t>\tau$  with $\tau \in \R$:
\be
\label{pde1}
{\frac {du}{dt}} +A u
=F(u) +g(t) + {\frac {dW}{dt}}
\quad {\mbox{with }} \ 
u(\tau) = u_\tau,
\ee
where
$A: D(A) \subseteq H \to H$ is a symmetric positive operator
with compact inverse, $F$ is a nonlinearity, $g$ is a time-dependent
external forcing, and 
  $W$ is a  $H$-valued Wiener process 
  on  a probability space
$(\Omega, \calf, P)$.

If $g$ does not depend on time, then \eqref{pde1}
is said to be an autonomous stochastic equation.
The existence of IMs for autonomous random systems
has been studied by experts in \cite{ben1, bru1, chu1,
chu2, chu3, chu4, gir1} and the references therein.
For such a system, an IM is a random set
$\calm =\{ \calm (\omega): \omega \in \Omega \}$
which is invariant, exponentially attracts all solutions,
and is given by the graph of a finite-dimensional Lipschitz
map.
In the present paper,  we want to  consider the
non-autonomous  case where  $g$ varies
in time.  It seems that there is no result
 reported in the literature on IMs  when $g$
 is  time-dependent.
As we will see in Section \ref{eim}, 
an IM for a non-autonomous stochastic equation
is parametrized  not only  by $\omega \in \Omega$,
but also by initial times  $\tau \in \R$.
More precisely, an IM in this case is a
finite-dimensional Lipschitz manifold 
$\calm =\{\calm (\tau, \omega):
\tau \in \R, \omega \in \Omega \}$
that is   invariant and exponentially attracts
all solutions.
We will first prove the existence of such IMs
for equation \eqref{pde1} in Section \ref{eim}
under the classical gap condition. We will also
prove that   these IMs must contain tempered
pullback random attractors if such an attractor
exists.
In addition, we will investigate the dependence of
IMs on the external forcing $g$. If $g$ is almost
periodic, we show the
random  IMs are pathwise almost periodic.
We will also prove   pathwise periodicity of 
IMs when $g$ is periodic in time.

We remark that    constructions of random IMs
are quite similar to random invariant manifolds
which have been investigated recently
in \cite{arn1, car1, dua1, dua2, garr1, lia1, lu1,moh1, wann1}.
Of course, for IMs, one must establish the exponential
attracting property which is also called the asymptotic
completeness of IMs. Interestingly enough, the proof
of asymptotic completeness  is often 
based on similar arguments for 
constructing 
IMs but under even more restrictive conditions.

In the next section, we review some concepts
of random dynamical systems and define IMs
for non-autonomous stochastic equations.
Section \ref{eim}
is devoted to the existence of IMs
for non-autonomous random systems
 under a gap condition.
 In the last section, we show the periodicity and almost
 periodicity of IMs when $g$ is periodic and almost
 periodic in time, respectively.

\section{Notation}\label{nota}
\setcounter{equation}{0}

       In this section, we introduce the concept
       of random IMs for non-autonomous
       stochastic equations.
       The reader is referred to
        \cite{ben1, bru1, chu1,
chu2, chu3, chu4, gir1}  on IMs
for autonomous  random systems.

       Assume  
       $X$ is a separable Banach  space
       with norm $\| \cdot \|$, and 
        $(\Omega, \calf, P, \{\theta_t\}_{t\in \R} )$
       is  a metric dynamical system as in \cite{arn1}.  
   Let $D=\{D(\tau, \omega): \tau \in \R, \omega \in \Omega \}$
   be a family
   of nonempty  bounded subsets of $X$.
   Then $D$ is said  to be tempered if     
   $ 
   \lim\limits_{t \to -\infty}
   e^{ct} \| D(\tau +t, \theta_t \omega )\|_X =0$
   for   every $c>0$, where 
   $\| D\|_X = \sup\limits_{x \in D} \| x \|$.
   A family $D=\{D(\tau, \omega): \tau \in \R, \omega \in \Omega \}$ 
   of nonempty closed subsets 
    is said to be measurable if it is measurable  in $\omega$ for
    each fixed $\tau $. The following is the definition of cocycle for
    non-autonomous random equations.

\begin{defn} \label{ds1}
Suppose   $\Phi$: $ \R^+ \times \R \times \Omega \times  X
\to X$  satisfies,    for any
  $\tau\in \R$,
  $\omega \in   \Omega $
  and    $t, s \in \R^+$,  
\begin{itemize}
\item [(i)]   $\Phi (\cdot, \tau, \cdot, \cdot): \R ^+ \times \Omega \times X
\to X$ is
 $(\calb (\R^+)   \times \calf \times \calb (X), \
\calb(X))$-measurable;

\item[(ii)]    $\Phi(0, \tau, \omega, \cdot) $ is the identity on $X$;

\item[(iii)]    $\Phi(t+ s, \tau, \omega, \cdot) =
 \Phi(t,  \tau +s,  \theta_{s} \omega, \cdot) 
 \circ \Phi(s, \tau, \omega, \cdot)$.
  \end{itemize}
 Such   $\Phi$ is called a cocycle  in $X$  over
$(\Omega, \calf, P,  \{\theta_t\}_{t \in \R})$.
If,  in addition,    $\Phi(t, \tau, \omega,  \cdot): X \to  X$ is continuous,
then we say $\Phi$ is a continuous cocycle in $X$.
\end{defn}
       
 \begin{defn}\label{defim}
   A measurable family $\calm
   =\{\calm (\tau, \omega): \tau \in \R,
   \omega \in \Omega \}$ of subsets of $X$
   is called a random  inertial manifold of $\Phi$ if the
   following conditions  (i)-(iii) are satisfied:
   \begin{itemize}
\item [(i)]  There exists a finite-dimensional subspace $X_1$
such that $X=X_1 \oplus X_2$ and there is a mapping
$m: \R \times \Omega \times X_1 \to X_2$ such that
for every $\tau \in \R$  and $\omega \in \Omega$,
\be\label{nota1}
\calm (\tau, \omega)
= \{ x + m(\tau, \omega) (x): x \in X_1 \}
\ee
where $m(\tau, \omega) (x)$ is measurable in $\omega\in \Omega$ and is
Lipschitz continuous in  $x \in X_1$;
 
\item[(ii)]   $\calm$ is  invariant:
$\Phi (t, \tau, \omega, \calm (\tau, \omega))
= \calm (\tau +t, \theta_t \omega )$ for all $t\in \R^+$,
$\tau \in \R$ and $\omega \in \Omega$;

   \item[(iii)]  $\calm$ is asymptotically complete:
   for every $\tau\in \R$,  $\omega \in \Omega$
   and $x\in X$, there exists $z=z(\tau, \omega, x)
   \in \calm (\tau, \omega)$ such that for all $t \ge 0$,
   $$
   \| \Phi (t, \tau, \omega, x)
   -\Phi (t,\tau, \omega, z) \|
   \le  c_1 e^{-c_2 t},
   $$
   where $c_1$  and $c_2$ are positive numbers
   depending on $\tau, \omega$  and $x$.
  \end{itemize}
 \end{defn}
 
 Next, we  define pathwise almost periodic random
 inertial manifolds.   
 Let
   $g: \R \to X$ be a continuous function.
   Then $g$  is  said to be  almost periodic 
   (see, e.g., \cite{yos1})
   if for every $\varepsilon>0$
    there exists a positive
   number $l = l(\varepsilon)$ 
   such that every interval of length $l$ contains a number $t_0$ for which
   $ \| g(t+t_0) - g(t) \| < \varepsilon$  for all $ t \in \R$.
   Motivated by almost periodic functions, we introduce the following concept.
 
 \begin{defn}\label{defperap}
 Let   $\calm
   =\{\calm (\tau, \omega): \tau \in \R,
   \omega \in \Omega \}$ be a random inertial
   manifold of $\Phi$ given by \eqref{nota1}.
  Then  $\calm$ is said to be pathwise almost periodic if 
   for every $\varepsilon>0$, there exists a positive
   number $l = l(\varepsilon)$ such that any interval of length
   $l$ contains a number $\tau_0$ such that
   $$
   \sup_{x\in X_1} \| m(\tau +\tau_0, \omega)(x)
   -m(\tau, \omega) (x) \| \le \varepsilon,
   \ \ \mbox{for all} \ \ \tau \in \R
   \  \mbox{and} \ \omega \in \Omega.
   $$
   
   If   there exists $T>0$ such that
  $m(\tau+T, \omega)(x) = m(\tau, \omega)(x)$
  for all $\tau \in \R$, $ \omega \in \Omega$
  and $x\in X_1$, then $\calm$ is called a
  pathwise periodic random inertial manifold with period $T$.
  \end{defn}
  
  It is evident that if $\calm$ is a $T$-periodic
  random inertial manifold in the sense of Definition \ref{defperap}, 
  then 
  $\calm(\tau +T, \omega) =\calm (\tau, \omega)$
  for all $\tau\in \R$ and $\omega \in \Omega$.
  Similarly, if $\calm$ is almost periodic, then 
 for every $\varepsilon>0$, there exists a positive
   number $l = l(\varepsilon)$ such that any interval of length
   $l$ contains a number $\tau_0$ such that
   for all $\tau \in \R$ and $\omega \in \Omega$,
   \be\label{nota2}
    d^H (\calm(\tau +\tau_0, \omega), \calm (\tau, \omega ))
    \le \varepsilon
    \ \ \mbox{and} \ \ 
  d^H (\calm(\tau , \omega), \calm (\tau +\tau_0, \omega) )
  \le \varepsilon,
  \ee
  where $d^H$ is the Hausdorff semi-distance between two subsets
  of $X$. Note that \eqref{nota2} actually indicates
  the Hausdorff distance of
  $\calm(\tau +\tau_0, \omega)$ and $ \calm (\tau, \omega)$
  is controlled by $\varepsilon$.

\section{Existence of Random Inertial Manifolds }
\label{eim}
\setcounter{equation}{0}

In this section,  we construct random inertial
manifolds for the stochastic equation
\eqref{pde1} in a 
  separable Hilbert space 
  $H$
  with norm
$\| \cdot \|$. 
Assume that
 $W$  in \eqref{pde1}
 is a two-sided $H$-valued Wiener process with covariance 
operator $Q$ of trace class on  a probability space
$(\Omega, \calf, P)$ where
$\Omega = \{ \omega \in C(\R, H): \omega (0) =0 \}$
with compact-open topology, $\calf$ is the Borel $\sigma$-algebra and 
$P$ the Wiener measure.
Suppose $A: D(A) \subseteq H \to H$ is a symmetric positive operator
with compact inverse. Then $H$ has an orthonormal basis 
$\{ e_n\}_{n=1}^\infty$ consisting of  eigenvectors of $A$:
$$
Ae_n = \lambda_n e_n,
\quad 0<\lambda_1 \le \lambda_2\le \dots \le \lambda_n \le \cdots
\quad \mbox{with } \  \lambda_n \to \infty.
$$
Let $F: D(A^\alpha) \to  H$   be a Lipschitz
nonlinearity for some $\alpha \in [0, {\frac 12})$; that is,
there exists a positive constant $L$ such that
\be\label{f1}
F(0) =0, \quad 
\| F(u_1) -F(u_2) \|
\le L \| A^\alpha (u_1 -u_2) \|
\ee
for all $u_1, u_2 \in D(A^\alpha)$. 
This implies that for all $u  \in D(A^\alpha)$,
\be\label{f2}
\| F(u) \| \le L \| A^\alpha u \|.
\ee
For the non-autonomous term $g$ in \eqref{pde1} we assume
that $g \in L^2_{loc} (\R, D(A^\alpha))$ and
\be
\label{g1}
\int^0_{-\infty} e^{\lambda_1 s} \| A^\alpha g(s) \| ds <\infty,
\ee
where $\lambda_1$ is the first eigenvalue of $A$.
These assumptions imply that 
  for every $\tau \in \R$,
  \be
\label{g2}
\int^0_{-\infty} e^{\lambda_1 s} \| A^\alpha g^\tau(s) \| ds <\infty,
\ee
where $g^\tau(\cdot)
=g(\cdot + \tau)$ is the translation of $g$ by $\tau$.

Given an integer $n \ge 1$, let
$P_n: H \to span\{e_1, \cdots, e_n \}$ be the orthogonal
projection and $Q_n = I-P_n$.  Then    we have
(see, e.g., \cite{chu1}) for all $u\in H$, 
\be
\label{pn1}
\|A^\alpha e^{-At} P_n u \|
\le \lambda_n^\alpha e^{-\lambda_n t} \| u \|,
\ \ \  t \le 0,
\ee
\be\label{qn1}
\| e^{-At} Q_n u \|
\le e^{-\lambda_{n+1} t } \| u \|,   \ \ \ t \ge 0,
\ee
and
\be\label{qn2}
\| A^\alpha e^{-At} Q_n u \|
\le
\left (
\alpha^\alpha t^{-\alpha} + \lambda^\alpha_{n+1} 
\right ) e^{-\lambda_{n+1} t} \| u\|,
\ \ \  t > 0.
\ee
Note that   $e^{-At}: P_n H \to P_n H$
is invertible for any $t\ge 0$, 
 and its inverse is denoted by $e^{At}$.
 This means  $e^{-At}:  P_n H \to P_n H$
 is defined for  all $t\in \R$.
 Let $\{\theta_t\}_{t\in \R}$ be a group of translations
 on $\Omega$ given by: 
 $\theta_t \omega  (\cdot)
 =\omega (\cdot +t) -\omega (t)
  \mbox{ for all }  \omega \in \Omega
 \  \mbox{ and }  t \in \R.$
 Then $(\Omega, \calf, P, \{\theta_t\}_{t\in \R})$
 is a metric dynamical system as in \cite{arn1}.
 Let $z: \Omega \to  D(A^\alpha)$ be  the unique stationary
 solution of  
 \be\label{a1}
  dz(\theta_t \omega) +A z(\theta_t \omega)
 =dW.
 \ee
 Then by \cite{chu1}, $z(\theta_t \omega)$ has a continuous
 version from $\R$ to $D(A^\alpha)$ for every fixed $\omega$,
 and is tempered in the sense that for every $c>0$ and $\omega\in
 \Omega$,
 $$ \lim_{t \to -\infty}
 e^{ct} \|  A^\alpha z(\theta_t \omega)\| =0.
 $$
 In terms of \eqref{a1}, we can transfer the stochastic equation
 \eqref{pde1} into a pathwise random one by introducing a new variable
 $v(t) = u(t) -z(\theta_t \omega)$.
 By \eqref{pde1} and \eqref{a1} we get
 \be\label{a20}
 {\frac {dv}{dt}} +Av
 =F(v+ z(\theta_t \omega ) )
 + g(t),
 \quad t>\tau, \ \ \ v(\tau) = v_\tau.
 \ee
 By \eqref{f1} and the Banach fixed point theory
 as in \cite{bru1}, one can prove that
 for every $v_\tau \in D(A^\alpha)$
 with $0\le \alpha <{\frac 12}$, problem
 \eqref{a20} has a unique mild solution in $C([\tau, \infty),
 D(A^\alpha))$, which is measure in $\omega$ and continuous
 in $v_\tau$ in $(D(A^\alpha))$.
 To indicate the dependence  on all related parameters, we write
 the solution  of \eqref{a20} as
 $v(t, \tau, \omega, g, v_\tau)$ which is given by
 $$
 v(t, \tau, \omega, g, v_\tau)
 =e^{-A (t-\tau)} v_\tau
 + \int_\tau^t e^{-A (t-s)} 
 \left (
 F(v(s) + z(\theta_s \omega) ) + g(s) 
 \right ) ds.
$$
 Let $\Psi: \R^+ \times \R \times \Omega 
 \times D(A^\alpha) \to D(A^\alpha)$ be a mapping given by,
 for all $t \in \R^+, \tau \in \R, \omega \in \Omega$
 and $v_\tau \in  D(A^\alpha)$,
 \be\label{a40}
 \Psi (t, \tau, \omega, v_\tau)
 =v(t+\tau, \tau, \theta_{-\tau} \omega, g, v_\tau)
 =v(t, 0, \omega, g^\tau, v_\tau),
 \ee
 where $g^\tau (\cdot) = g(\cdot +\tau)$ as usual.
 Then we find that $\Psi$ is a continuous cocycle
 over $(\Omega, \calf, P, \{\theta_t\}_{t\in \R} )$.
 Note that if $v$ is a solution of \eqref{a20}, then
 the process
 \be\label{a3}
 u(t, \tau, \omega, g, u_\tau)
 =v(t, \tau, \omega, g, v_\tau)
 +z(\theta_t \omega)
 \quad \mbox{with} \ \ 
 u_\tau = v_\tau + z(\theta_\tau \omega)
 \ee
 is a mild solution of the stochastic equation
 \eqref{pde1}. Based on this fact, we can define
 a continuous cocycle $\Phi$ for \eqref{pde1} by
 \be\label{a41}
 \Phi (t,\tau, \omega, u_\tau)
 =u(t+\tau, \tau, \theta_{-\tau} \omega, g, u_\tau)
 \ee
 for  all $t\in \R^+, \tau\in \R, \omega \in \Omega$
 and $u_\tau \in D(A^\alpha)$.
 By \eqref{a40}-\eqref{a41}  we have
 \be\label{a42}
 \Phi (t, \tau, \omega, u_\tau)
 = \Psi (t, \tau, \omega, u_\tau -z(\omega))
 + z(\theta_t \omega).
 \ee
 We often need to replace $g$ by $g^\tau$
 on the right-hand side of \eqref{a20} and conside the
 equation
  \be\label{a50}
 {\frac {dv}{dt}} +Av
 =F(v+ z(\theta_t \omega ) )
 + g^\tau (t),
 \quad t> r, \ \ \ v(r) = v_r.
 \ee
 The solution of \eqref{a50}  is given by, for $t \ge r$,
 \be\label{a52}
 v(t, r, \omega, g^\tau, v_r)
 =e^{-A (t-r)} v_r
 + \int_r^t e^{-A (t-s)} 
 \left (
 F(v(s) + z(\theta_s \omega) ) + g(s +\tau) 
 \right ) ds.
 \ee
 Note that the solution of  \eqref{a50} is defined only for
 forward time $t\ge r$ in general. But we need to consider 
 backward  solutions when constructing inertial manifolds.
For that purpose, we introduce the following concept
of solutions defined on $(-\infty, 0]$.

\begin{defn}
\label{defb1}
Given $\tau \in \R$ and $\omega \in \Omega$,
a continuous mapping $\xi: (-\infty, 0]
\to D(A^\alpha)$ is called  a mild solution of \eqref{a50}
on $(-\infty, 0]$ if  
$v(t,r,\omega, g^\tau, \xi(r)) =\xi (t)$ 
for all $r\le t\le 0$,
where
$v(t,r,\omega, g^\tau, \xi(r))$ is the unique solution
of \eqref{a50} with  initial value $\xi(r)$.
\end{defn}
 
 We now  assume that
 there is $n\in \N$ such that
 \be\label{b1}
 \lambda_{n+1} -\lambda_n
 \ge {\frac {2L}k} \left (
 \lambda^\alpha_{n+1} +\lambda^\alpha_n
 +c_\alpha (\lambda_{n+1} -\lambda_n)^\alpha
 \right )
 \quad \mbox{for   some } \  k\in (0, 1),
 \ee
 where $c_\alpha$ is a nonnegative
  number given by
 $$
 c_\alpha
 = \alpha^\alpha\int_0^\infty
 s^{-\alpha} e^{-s} ds
 \quad \mbox{if }  \ \alpha>0
 \quad \mbox{ and } \ c_\alpha = 0
 \quad \mbox{if } \ \alpha =0.
 $$
 For convenience, we set
 \be\label{b3}
 \mu = \lambda_n + {\frac {2L}k} \lambda_n^\alpha.
 \ee
 Then by \eqref{b1} we have
 $\mu \in (\lambda_n, \lambda_{n+1})$. Let $\cals$
 be the Banach space defined by
 $$
 \cals =\{ \xi\in C((-\infty, 0], D(A^\alpha)): \sup_{t\le 0}
 e^{\mu t} \| A^\alpha  \xi (t) \| < \infty  \}
 $$
 with norm $\| \xi \|_{\cals}
 = \sup\limits_{t\le 0}
 e^{\mu t} \| A^\alpha  \xi (t)\|$.
 Given $\tau \in \R$  and $\omega \in \Omega$, denote by
 $$\calm (\tau, \omega)
 =\{ \xi (0): \xi \in \cals
 \ \mbox{and} \  \mbox{is a
 mild solution of \eqref{a50} on } (-\infty, 0]  \ 
 \mbox{by Definition} \ \ref{defb1}\}
 $$
 \be\label{b4}
 =\{ \xi (0): \xi \in \cals
 \ \mbox{and} \ v(t,r, \omega, g^\tau, \xi(r)) = \xi (t)
 \  \mbox{for all} \  r\le t \le 0 \}.
 \ee
 We will show $\calm
 =\{ \calm (\tau, \omega): \tau \in \R, \omega \in \Omega \}$
 is an inertial manifold of \eqref{a50} for which we need:
 
 \begin{lem}
 \label{lemb1}
 Suppose  \eqref{f1}, \eqref{g1}  and \eqref{b1}-\eqref{b3}
 hold, and $\xi \in  \cals$.
 Then $\xi$ is a mild solution of  \eqref{a50} on $(-\infty, 0]$
 in the  sense of Definition \ref{defb1}
 if and only if there exists $x \in P_n H$ such that for all
 $t \le 0$,
 $$
 \xi (t)
 = e^{-At} x
 -\int_t^0 e^{-A (t-s)} P_n 
 \left ( F( \xi (s) + z(\theta_s \omega))
 + g(s+\tau) \right ) ds
 $$
 \be\label{b5}
 +
 \int^t_{-\infty}  e^{-A (t-s)} Q_n 
 \left ( F( \xi (s) + z(\theta_s \omega))
 + g(s+\tau) \right ) ds.
 \ee
 \end{lem}
 \begin{proof}
 Given $\xi \in \cals$ and $t \le 0$, it is evident that
 the first integral on $(t,0)$ in \eqref{b5} is well-defined.
 We now  prove the second integral on $(-\infty, t)$ exists.
 By \eqref{g2},  \eqref{qn1}
 and $\lambda_{n+1} \ge \lambda_1$  we have
  \be\label{b20}
 \int^t_{-\infty} \| A^\alpha e^{-A (t-s)} Q_n g(s+\tau) \| ds
 \le
  \int^t_{-\infty} e^{-\lambda_{n+1} (t-s)}
   \| A^\alpha   g(s+\tau) \| ds <\infty.
  \ee
  By \eqref{f2}  and \eqref{qn2} we have
  $$
  \int^t_{-\infty} \| A^\alpha  e^{-A (t-s)} Q_n 
  F( \xi (s) + z(\theta_s \omega) )  \| ds
  $$
  $$
  \le
  L\int^t_{-\infty}
  \left (
  {\frac {\alpha^\alpha}{ (t-s)^{\alpha} }} +\lambda^\alpha_{n+1}
  \right ) e^{-\lambda_{n+1} (t-s)}
  \| \xi (s) +  z(\theta_s \omega)\|_{D(A^\alpha)} ds
  $$
    \be\label{b21}
  \le
  L\| \xi + z(\theta_s\omega)\|_{\cals}
  \int^t_{-\infty}
  \left (
  {\frac {\alpha^\alpha}{ (t-s)^{\alpha} }} +\lambda^\alpha_{n+1}
  \right )  e^{-\lambda_{n+1} t} e^{(\lambda_{n+1} -\mu)  s}ds
  < \infty,
\ee
where the last inequality follows from
  $\mu <\lambda_{n+1}$ and the temperedness of  $z$.
  By \eqref{b20} and \eqref{b21}, the integrals in \eqref{b5}
  we well-defined in $D(A^\alpha)$ for all $\xi \in \cals$ and $t \ge 0$.
  If $\xi \in \cals$ is a solution of \eqref{a50} on $(-\infty, 0]$, then by
  \eqref{a52} and Definition \ref{defb1} we  get  for all
  $r<t \le 0$,
  $$\xi (t)
  =e^{-A (t-r)}\xi (r)
  +\int_r^t e^{-A (t-s)}
  ( F(\xi (s) + z(\theta_s \omega)) + g(s+\tau) ) ds,
  $$
  which implies for all  $r< t \le 0$,
  \be\label{b6}
  P_n \xi (t)
  =e^{-A (t-r)}P_n \xi (r)
  +\int_r^t e^{-A (t-s)} P_n
  ( F(\xi (s) + z(\theta_s \omega)) + g(s+\tau) ) ds,
 \ee
  and
  \be\label{b7}
  Q_n \xi (t)
  =e^{-A (t-r)}Q_n \xi (r)
  +\int_r^t e^{-A (t-s)} Q_n
  ( F(\xi (s) + z(\theta_s \omega)) + g(s+\tau) ) ds.
\ee
  We get from \eqref{b6} for $t =0$
  $$
  P_n \xi (0)
  =e^{A  r }P_n \xi (r)
  +\int_r^0  e^{ A  s} P_n
  ( F(\xi (s) + z(\theta_s \omega)) + g(s+\tau) ) ds
 $$
 and hence
  \be\label{b8}
   e^{-At} P_n \xi (0)
  =e^{-A (t-r)}P_n \xi (r)
  +\int_r^0 e^{-A (t-s)} P_n
  ( F(\xi (s) + z(\theta_s \omega)) + g(s+\tau) ) ds.
 \ee
 It follows   from \eqref{b6}  and \eqref{b8} that
  \be\label{b9}
  P_n \xi (t)
  =e^{-A t  }P_n \xi (0 )
  -   \int_t^0   e^{-A (t-s)} P_n
  ( F(\xi (s) + z(\theta_s \omega)) + g(s+\tau) ) ds.
 \ee
 On the other hand,  since 
 $\mu\in (\lambda_n, \lambda_{n+1} )$,
 by \eqref{qn1} we have  for $ r<t \le 0$,
 \be\label{b10}
 \| e^{-A(t-r)} Q_n \xi (r) \|_{D(A^\alpha)}
 \le e^{-\lambda_{n+1} (t-r)} \| \xi (r) \|_{D(A^\alpha)}
 \le e^{-\lambda_{n+1} t}
 e^{(\lambda_{n+1} -\mu )r }\| \xi \|_{\cals}
 \to 0
 \ee
 as  $ r \to -\infty$.
 Taking the limit of \eqref{b7} as $r \to -\infty$, by
 \eqref{b9} and the fact
 $\xi (t)
 =P_n \xi (t)
 +Q_n \xi (t)$ we find that $\xi$ satisfies \eqref{b5}
 with $x= P_n \xi (0)$.
 
 Suppose now $\xi \in \cals$  satisfies \eqref{b5}. Then by
 simple calculations, one can verify that for all
 $r\le t \le 0$,
 $v(t,r,\omega, g^\tau, \xi (r))
 =\xi (t)$ and hence $\xi$ is  a solution of  \eqref{a50}
 on $(-\infty, 0]$.
 \end{proof}
 
 Next, we will find all solutions $\xi \in \cals$ of 
 \eqref{a50}  on $(-\infty, 0]$ in order to characterize
 the structure of 
 $\calm$ given by \eqref{b4}.
 By Lemma \ref{lemb1}, we only need 
 to find all $\xi \in \cals$ satisfying \eqref{b5}.
 Given  $\tau \in \R$, 
 $\omega \in \Omega$, $x \in P_nH$ and $\xi \in \cals$, denote by
 $$
 \cali (\xi, x, \omega,  \tau) (t)
 = e^{-At} x
 -\int_t^0 e^{-A (t-s)} P_n 
 \left ( F( \xi (s) + z(\theta_s \omega))
 + g(s+\tau) \right ) ds
 $$
 \be\label{c1}
 +
 \int^t_{-\infty}  e^{-A (t-s)} Q_n 
 \left ( F( \xi (s) + z(\theta_s \omega))
 + g(s+\tau) \right ) ds.
 \ee
 Then $\xi$ satisfies \eqref{b5} if and only if
 $\xi$   is a  fixed point of $\cali$ in $\cals$.
 We first show $\cali$ maps $\cals$ into itsself.
 
 \begin{lem}
 \label{lemb2}
 Suppose  \eqref{f1}, \eqref{g1}  and \eqref{b1}-\eqref{b3}
 hold. Then for every fixed $x\in P_n H$,
 $\omega \in \Omega$ and $\tau \in \R$,
 $\cali(\cdot, x, \omega, \tau): \cals
 \to \cals$ is well-defined.
 \end{lem}
 
 \begin{proof}
 By \eqref{c1}  and \eqref{pn1}-\eqref{qn2}
 we have for each $\xi \in \cals$ and $t\le 0$,
 $$
 \|\cali (\xi, x, \omega, \tau) (t)\|_{D(A^\alpha)}
 \le  e^{-\lambda_n t}  \| A^\alpha x \|
  + \int_t^0 
  e^{- \lambda_n (t-s)} 
 \left (\lambda_n^\alpha \| F( \xi  + z(\theta_s \omega))\|
 + \| A^\alpha g(s+\tau)\| \right ) ds
 $$
\be\label{c2}
 +
 \int^t_{-\infty}  
 ({\frac {\alpha^\alpha}{(t-s)^\alpha}} + \lambda_{n+1}^\alpha  )
 e^{- \lambda_{n+1} (t-s)} 
 \| F( \xi  + z(\theta_s \omega)) \| ds
 +\int_{-\infty}^t e^{-\lambda_{n+1} (t-s)} \| A^\alpha g(s+\tau)\| ds.
\ee
 By \eqref{f2}  and \eqref{b3} we have for all $t \le 0$,
 $$
  e^{\mu t}
  \int_t^0 
  e^{- \lambda_n (t-s)} 
 \left (\lambda_n^\alpha \| F( \xi  + z(\theta_s \omega))\|
 + \| A^\alpha g(s+\tau)\| \right ) ds
 $$
 $$
 \le L\lambda^\alpha_n \| \xi + z\|_\cals
 \int_t^0 e^{(\mu -\lambda_n) (t-s)} ds
 +\int_t^0 e^{ (\mu-\lambda_n)t} 
 e^{\lambda_n s} \| A^\alpha g(s+\tau)\|ds
 $$
 \be\label{c5}
 \le
 {\frac {\lambda_n^\alpha L}{\mu-\lambda_n}} \| \xi + z\|_\cals
 + \int_t^0 e^{\lambda_1 s} \| A^\alpha g(s+\tau) \| ds.
 \ee
 Similarly, for all $t \le 0$ we get
 $$
 e^{\mu t} \int^t_{-\infty}  
 ({\frac {\alpha^\alpha}{(t-s)^\alpha}} + \lambda_{n+1}^\alpha  )
 e^{- \lambda_{n+1} (t-s)} 
 \| F( \xi (s) + z(\theta_s \omega)) \| ds
 $$
 $$
 \le
 L \|\xi +  z \|_\cals
  \int^t_{-\infty}  
 ({\frac {\alpha^\alpha}{(t-s)^\alpha}} + \lambda_{n+1}^\alpha  )
 e^{(\mu - \lambda_{n+1}) (t-s)}   ds
 $$
 \be\label{c4}
 \le L (\lambda_{n+1} -\mu)^{-1}
 \left ( \lambda_{n+1}^\alpha + c_\alpha (\lambda_{n+1} -\mu)^\alpha
 \right )
 \|\xi + z \|_\cals.
 \ee
 In addition, by \eqref{g2} and  \eqref{b3} we have for $t \le 0$,
 $$
 e^{\mu t}
 \int_{-\infty}^t e^{-\lambda_{n+1} (t-s)} \| A^\alpha g(s+\tau)\| ds
 \le
 \int_{-\infty}^t e^{(\mu -\lambda_{n+1}) (t-s)}
 e^{\mu s}  \| A^\alpha g(s+\tau)\| ds
 $$
 \be\label{c3}
 \le
 \int_{-\infty}^t  
 e^{\lambda_1  s}  \| A^\alpha g(s+\tau)\| ds
 < \infty.
\ee
Since $\lambda_{n+1}- \mu \le \lambda_{n+1} -\lambda_n$
by \eqref{b3},  it follows from \eqref{c2}-\eqref{c3} that
$$
\| \cali (\xi, x, \omega, \tau)\|_\cals
\le
\left (
{\frac {\lambda_n^\alpha}{\mu - \lambda_n}}
+ {\frac {\lambda_{n+1}^\alpha + c_\alpha (\lambda_{n+1}
-\lambda_n)^\alpha}{\lambda_{n+1} -\mu}}
\right )\| \xi +z \|_\cals  + \|A^\alpha x\|
$$
\be\label{b6-1*}
+ \int_{-\infty}^0  
 e^{\lambda_1  s}  \| A^\alpha g(s+\tau)\| ds
 \le
 k \|\xi  + z\|_\cals 
 + \|A^\alpha x\| 
 +
 \int_{-\infty}^0  
 e^{\lambda_1  s}  \| A^\alpha g(s+\tau)\| ds,
 \ee
   where the last inequality follows from 
   \eqref{b1}. Therefore,  we   have
   $ \cali(\xi, x, \omega, \tau) \in \cals$.
 \end{proof}
 
 We now establish existence of fixed points of $\cali$ in $\cals$.

\begin{lem}
\label{lemb3}
Suppose  \eqref{f1}, \eqref{g1}
and \eqref{b1}-\eqref{b3} hold.
Then for every $x \in P_n H$,
$\omega \in \Omega$ and $\tau \in \R$,
$\cali (\cdot, x, \omega, \tau): \cals
\to \cals$ has a unique fixed point
which is Lipschitz continuous in $x\in P_n H$.
\end{lem}

\begin{proof}
Given $\xi_1, \xi_2 \in  \cals$, following the proof of Lemma \ref{lemb2},
one can verify 
\be\label{c8}
\| \cali (\xi_1, x, \omega, \tau)
-\cali (\xi_2, x, \omega, \tau)\|_\cals
\le k \| \xi_1 -\xi_2 \|_\cals.
\ee
Since $k\in (0,1)$, by \eqref{c8} we find that
$\cali (\cdot, x, \omega, \tau)$ has a unique fixed
point $\xi^*(x, \omega, \tau)$ in $\cals$. 
For every $t\le 0$,  $\xi^*(x, \cdot, \tau) (t):
\Omega \to D(A^\alpha)$ is measurable   since
it is a limit of iterations of measurable functions
staring at zero.
 In addition, by \eqref{b6-1*} we have
\be\label{c8_a0}
(1-k) \|\xi^*(x, \omega,  \tau) \|_\cals
\le  k  \| z(\theta_{s}  \omega) \|_\cals
+ \| A^\alpha x \|
+ 
 \int_{-\infty}^0  
 e^{\lambda_1  s}  \| A^\alpha g(s+\tau)\| ds.
\ee
 We now prove  the continuity of $\xi^*$ in $x \in P_n H$.
 Let $\xi_1^*
 = \xi^*(x_1, \omega, \tau)$
 and $\xi_2^*=  \xi^*(x_2, \omega,  \tau)$. 
   By  \eqref{b3},   \eqref{c1} and \eqref{c8}
 we have
 $$
 \| \xi_1^* -\xi_2^*\|_\cals
 = \| \cali (\xi_1^*, x_1, \omega, \tau)
 -\cali (\xi_2^*, x_2, \omega, \tau) \|_\cals
 $$
 $$\le
 \| \cali (\xi_1^*, x_1, \omega, \tau)
 -\cali (\xi_1^*, x_2, \omega, \tau) \|_\cals
 +
 \| \cali (\xi_1^*, x_2, \omega, \tau)
 -\cali (\xi_2^*, x_2, \omega, \tau) \|_\cals
 $$
 \be\label{c8_a1}
 \le
 \sup_{t\le 0} e^{ (\mu -\lambda_n) t}
 \|x_1-x_2\|_{D(A^\alpha)}
 + k  \| \xi_1^* -\xi_2^* \|_\cals
 \le
 \|x_1-x_2\|_{D(A^\alpha)}
 + k  \| \xi_1^* -\xi_2^* \|_\cals.
 \ee
 This implies 
   the Lipschitz continuity of $\xi^*$ in $x$.
\end{proof}

By the fixed point $\xi^*$ established by Lemma \ref{lemb3},
for each $\tau \in \R$   and $\omega \in \Omega$, we can define
a mapping
$m(\tau, \omega): P_n H \to Q_n D(A^\alpha)$ by
\be\label{c15}
m(\tau, \omega) (x)
=Q_n \xi^* (x, \omega, \tau) (0)
=  \int^0_{-\infty}  e^{A s } Q_n 
 \left ( F( \xi^*(x, \omega, \tau) (s) + z(\theta_s \omega))
 + g(s+\tau) \right ) ds
\ee
 for $x \in P_n H$.  It follows from \eqref{c8_a1} that
 \be\label{lipm}
 \| m(\tau, \omega) (x_1) -
 m(\tau, \omega) (x_2)\|_{D(A^\alpha)}
 \le {\frac 1{1-k}} \|x_1-x_2\|_{D(A^\alpha)}.
 \ee
 By \eqref{c8_a0} and  \eqref{c15} we get
 \be\label{c16}
 \|m(\tau, \omega) (x) \|_{D(A^\alpha)}
\le  {\frac k {1-k}} \| z(\theta_{s}  \omega) \|_\cals
+ {\frac 1 {1-k}} \| A^\alpha x \|
+ {\frac 1 {1-k}} 
 \int_{-\infty}^0  
 e^{\lambda_1  s}  \| A^\alpha g(s+\tau)\| ds.
\ee
 By  \eqref{b4}, \eqref{c15} and Lemma \ref{lemb1} we
 find that
 for all $\tau \in \R$ and $\omega \in \Omega$,
 \be\label{c17}
 \calm (\tau, \omega)
 =\{ \xi^*(x, \omega, \tau)  (0) : x\in P_n H \}
 = \{ x + m(\tau, \omega)(x): x \in P_n H \}.
 \ee
 Since  $m(\tau, \omega)(x)$ is measurable in $\omega$
 and continuous in $x$,  the measurability of
 $\calm(\tau, \cdot)$ follows.
 
 \begin{lem}
 \label{lemb4}
 Suppose \eqref{f1}, \eqref{g1} and \eqref{b1}-\eqref{b3} hold.
 Then $\calm =\{\calm(\tau, \omega): \tau \in \R,
 \omega \in \Omega \}$ given by
 \eqref{b4} is a Lipschitz invariant manifold of \eqref{a50}.
 \end{lem}
 
 \begin{proof}
 Given $\tau \in \R$ and $\omega \in \Omega$,
 by \eqref{c17},
 $\calm (\tau, \omega)$ is the graph of 
 $m(\tau, \omega)$ which is Lipschitz continuous
 by Lemma \ref{lemb3}.
 Next, we show the invariance of $\calm$. Given $t>0$
 and $v_0 \in \calm (\tau, \omega)$, by \eqref{b4} there 
 exists $\xi \in \cals$ such that $\xi (0) =v_0$ and $\xi$ is
 a solution of \eqref{a50} on $(-\infty, 0]$. 
Let
 \be\label{c23}
 \widetilde{\xi} (r)
 =
 \left \{
 \begin{array}{ll}
 v(r+t, 0, \omega, g^\tau, \xi (0))  & \ \ \mbox{ if  } \ -t \le r \le 0;\\
 \xi (r+t) & \ \  \mbox{ if } \  \  r<-t.
 \end{array}
 \right .
 \ee
 Note that  $\widetilde{\xi} \in \cals$
 since $\xi \in \cals$.  By straightforward  calculations,
 one can verify, for all $r_2 \le r_1 \le 0$,
 \be\label{c24}
 v(r_1, r_2, \theta_t \omega, g^{\tau +t},  \widetilde{\xi}(r_2))
 = \widetilde{\xi} (r_1),
 \ee
 where 
 $v(r_1, r_2, \theta_t \omega, g^{\tau +t},  \widetilde{\xi}(r_2))$
 is given by \eqref{a52}
 with $\omega$ and $g^\tau$ replaced
 by $\theta_t \omega$ and $g^{\tau +t}$, respectively.
 By \eqref{b4} and \eqref{c24} we find that
 $ \widetilde{\xi} (0) \in \calm (\tau +t, \theta_t \omega)$,
 which along with \eqref{a40}
 and  \eqref{c23} indicates
  that
  $\Psi (t, \tau, \omega, v_0)
  =v(t, 0, \omega, g^\tau, v_0)
  \in \calm (\tau + t,  \theta_t \omega)$
  for all $v_0 \in \calm (\tau, \omega)$,
  and hence
  \be\label{c20}
  \Psi (t, \tau, \omega, \calm (\tau, \omega) )
  \subseteq 
 \calm  (\tau + t,  \theta_t \omega),
 \ \ \mbox{for all } \ t \ge 0.
 \ee
 On the other hand, for every $v_0 \in \calm
 (\tau +t, \theta_t \omega)$, by \eqref{b4} there exists
 $\xi \in \cals$ such that $\xi (0) =v_0$ and
 for all $s_2\le s_1 \le 0$,
 \be\label{c31}
 v(s_1, s_2, \theta_t \omega, g^{\tau +t}, \xi (s_2))
 = \xi (s_1).
 \ee
 Let $\widetilde{\xi}: (-\infty, 0] \to  D(A^\alpha)$ be given by
 \be\label{c32}
 \widetilde{\xi} (r) = \xi (r-t) \ \ \mbox{for all } \  r \le 0.
 \ee
 Then $\widetilde{\xi} \in \cals$, and
 by \eqref{c31}-\eqref{c32},  we have
 for $t\ge 0$ and  for all $r_2 \le r_1 \le 0$,
 \be\label{c33}
 \widetilde{\xi} (r_1)
 =\xi (r_1 -t)
 =v(r_1 -t, r_2 -t, \theta_t \omega, g^{\tau +t}, \xi(r_2 -t))
 =v(r_1, r_2, \omega, g^\tau, \widetilde{\xi} (r_2) ).
 \ee
 By \eqref{c33}, \eqref{c32}    and \eqref{b4} we find that
 $\widetilde{\xi} (0) = \xi (-t) \in \calm (\tau, \omega)$. 
 By \eqref{a40} and \eqref{c31} we get
\be\label{c34}
 \Psi (t, \tau, \omega, \xi (-t))
 =v(t+\tau, \tau, \theta_{-\tau} \omega, g, \xi (-t))
 =v(0, -t, \theta_t \omega, g^{\tau +t}, \xi (-t))=\xi (0).
\ee
Since $\xi (-t) \in \calm (\tau, \omega)$
and $\xi (0) =v_0$ where $v_0$ is an arbitrary given
point in $\calm (\tau +t,  \theta_t \omega)$, by
\eqref{c34} we obtain
\be\label{c35}
 \calm  (\tau + t,  \theta_t \omega)
  \subseteq 
  \Psi (t, \tau, \omega, \calm (\tau, \omega) ),
 \ \ \mbox{for all } \ t \ge 0.
 \ee
 Then the invariance of $\calm$ follows from \eqref{c20} and \eqref{c35}
 immediately.
 \end{proof}
 
 The next result is concerned with the exponential attraction property
 of $\calm$.

 \begin{lem}
 \label{leme1}
 Suppose \eqref{f1}, \eqref{g1} and \eqref{b1}-\eqref{b3} hold
 with $k\in (0, \frac 12)$.
 Then for every  $\tau \in \R$,   $\omega \in \Omega$
 and $v_0 \in D(A^\alpha)$, there
 exists a random variable $v_0^* (\tau, \omega) \in \calm
 (\tau, \omega)$ such that for all $t\ge 0$, 
 $$
 \|\Psi(t,\tau, \omega, v_0^*)
 -\Psi (t, \tau, \omega, v_0)\|_{D(A^\alpha)}
 \le
 {\frac 1{1-\delta}} e^{-\mu t}
 \| Q_n v_0 - m(\tau,\omega) (P_n v_0) \|_{D(A^\alpha)},
 $$
 where $\delta =k + {\frac k{2 -2k}} \in (0,1)$
 for $0<k<{\frac 12}$.
 \end{lem}
 
 \begin{proof}
 We argue  as in \cite{chu1}.
 Let 
 $\cals^+ =\{ \xi \in C([0, \infty), D(A^\alpha)):
 \sup\limits_{t\ge 0}
 e^{\mu t} \|A^\alpha \xi (t) \| <\infty \}$ with norm
 $\|\xi\|_{\cals^+}
 =\sup\limits_{t \ge 0} e^{\mu t} \| A^\alpha \xi (t) \|$.
 We   will  find  $v_0^*\in \calm (\tau, \omega) $
 such that
 $ v(t, 0, \omega, g^\tau, v_0^*)
 - v(t,0, \omega, g^\tau, v_0) \in \cals^+$.
 To that end, we need to solve the 
  the equation
  $$
 \xi (t)
 = e^{-At} y_0
 +  \int^t_0  e^{-A (t-s)} Q_n 
 \left ( F( \xi (s) + z(\theta_s \omega) + v(s) )
  - F(v(s) + z(\theta_s \omega)) \right ) ds
 $$
 \be\label{d1}
 -
 \int_t^\infty  e^{-A (t-s)} P_n 
 \left ( F( \xi (s) + z(\theta_s \omega) +  v(s)  )
 -F(v(s) + z(\theta_s \omega) ) \right ) ds,
 \ee
 where $v(s) = v(s, 0, \omega, g^\tau, v_0)$ and $y_0$ is 
 determined by
 \be\label{d2}
 y_0 = 
 -Q_n v_0
 + m(\tau, \omega)
 \left (
 P_n v_0
 -\int_0^\infty e^{As} P_n
 (F(\xi +v+ z(\theta_s\omega))-
 F(v+ z(\theta_s \omega))) ds
 \right ).
 \ee
 Given $\xi \in \cals^+$,  denote the right-hand side of
 \eqref{d1} by $\cali^+ (\xi)$. We will  find a fixed point
 of $\cali^+$ in $\cals^+$.  By \eqref{pn1}-\eqref{qn2}
 and \eqref{f1} we get for
 $ t\ge 0$ and $\xi \in \cals^+$,
 $$
 e^{\mu t} \|A^\alpha  \cali^+ (\xi) (t)\|
 \le e^{ (\mu -\lambda_{n+1}) t }  \|A^\alpha y_0\|
 +
  L e^{\mu t}\int_t^\infty
 \lambda_n^\alpha e^{-\lambda_n (t-s)}\| A^\alpha \xi (s) \| ds
 $$
 $$
 +
 Le^{\mu t}\int_0^t
 \left (
 {\frac {\alpha^\alpha}{(t-s)^\alpha}}
 +\lambda_{n+1}^\alpha
 \right ) e^{-\lambda_{n+1} (t-s)} \| A^\alpha \xi (s) \| ds
 $$
 \be\label{d3}
  \le   \|A^\alpha y_0\|
  + L\|\xi\|_{\cals^+}
  \left (
   {\frac {\lambda_n^\alpha}{\mu -\lambda_n}}
   +{\frac {\lambda_{n+1}^\alpha}{\lambda_{n+1} -\mu}}
   + c_\alpha (\lambda_{n+1} - \mu)^{\alpha -1}
  \right ).
 \ee
 By  \eqref{d2},  \eqref{lipm}, \eqref{pn1} and
 \eqref{f1}  we get 
 $$
 \| A^\alpha y_0  \|
 \le \| -Q_n v_0 + m(\tau, \omega) (P_n v_0)\|_{D(A^\alpha)}
 $$
 $$
 +
 \|m(\tau, \omega) (P_n v_0) - m(\tau, \omega)
 \left (
 P_n v_0
 -\int_0^\infty e^{As} P_n
 (F(\xi +v+ z(\theta_s\omega))-
 F(v+ z(\theta_s \omega))) ds
 \right ) \|_{D(A^\alpha)}
 $$
 $$
 \le \| Q_n v_0 - m(\tau, \omega) (P_n v_0)\|_{D(A^\alpha)}
 + {\frac 1{1-k}}
  \int_0^\infty \|e^{As} P_n
 (F(\xi +v+ z)-
 F(v+ z) ) ds \|_{D(A^\alpha)}
 $$
 $$
 \le \| Q_n v_0 - m(\tau, \omega) (P_n v_0)\|_{D(A^\alpha)}
 + {\frac {L\lambda^\alpha_n}{1-k}}
  \int_0^\infty    e^{(\lambda_n -\mu) s} e^{\mu s} \|A^\alpha \xi(s)\|ds
 $$
  \be\label{d4}
 \le \| Q_n v_0 - m(\tau, \omega) (P_n v_0)\|_{D(A^\alpha)}
 + {\frac {L\lambda^\alpha_n}{(1-k) (\mu - \lambda_n)}}
 \| \xi \|_{\cals^+}.
 \ee
 It follows from \eqref{d3}-\eqref{d4} and \eqref{b1}-\eqref{b3} that
 \be\label{d5}
 \| \cali^+ (\xi)\|_{\cals^+}
 \le 
\| Q_n v_0 - m(\tau, \omega) (P_n v_0)\|_{D(A^\alpha)}
$$
$$
+
  L\|\xi\|_{\cals^+}
  \left (
  {\frac {(2-k)\lambda_n^\alpha}{(1-k)(\mu-\lambda_n)}}
   +{\frac {\lambda_{n+1}^\alpha}{\lambda_{n+1} -\mu}}
   + c_\alpha (\lambda_{n+1} - \mu)^{\alpha -1}
  \right )
  $$
  $$
  \le
\| Q_n v_0 - m(\tau, \omega) (P_n v_0)\|_{D(A^\alpha)}
+ \delta \| \xi \|_{\cals^+},
\  \  \mbox{with} \ \ 
\delta = k+ {\frac {k}{2-2k}}.
  \ee
  This shows that $\cali^+$ maps $\cals^+$ into itself.
  Note that for 
  $\xi_1, \xi_2 \in \cals^+$,   we have
  $$
  \| \cali^+ (\xi_1)
  -\cali^+(\xi_2) \|_{D(A^\alpha)}
  \le \| e^{-At}y_{0,1} -e^{-At} y_{0,2} \|_{D(A^\alpha)}
  $$
  $$
  + \int_0^t
  \| e^{-A(t-s)} Q_n (F(\xi_1 +z +v) -F(\xi_2 +z + v))\|_{D(A^\alpha)} ds
  $$
  $$
  + \int_t^\infty
  \| e^{-A(t-s)} P_n (F(\xi_1 +z +v) -F(\xi_2 +z + v))\|_{D(A^\alpha)} ds.
  $$
  Following the proof   of  \eqref{d5} we can get
   \be\label{d6}
   \| \cali^+(\xi_1)
   -\cali^+(\xi_2) \|_{\cals^+}
   \le \delta \| \xi_1 -\xi_2 \|_{\cals^+}
   \ee
   where $\delta$ is given in \eqref{d5}
   and $\delta \in(0,1)$ for $k\in (0, {\frac 12})$.
   Since $\cali^+: \cals^+ \to \cals^+$ is a contraction by \eqref{d6},  it has 
   a unique fixed point $\xi$ in $\cals^+$
   which satisfies \eqref{d1}-\eqref{d2}.
   This fixed point  is measurable  since it can be obtained
   by   a 
   limit of iterations of measurable functions.
   Further, by \eqref{d5} we obtain
   $$\| \xi\|_{\cals^+}
   =\| \cali^+\|_{\cals^+}
   \le   
\| Q_n v_0 - m(\tau, \omega) (P_n v_0)\|_{D(A^\alpha)}
+ \delta \| \xi \|_{\cals^+}
$$
and hence the fixed point $\xi$ satisfies, for all $t\ge 0$,
\be\label{d7}
\| \xi (t)  \|_{D(A^\alpha)} \le  {\frac 1{1-\delta}} e^{-\mu t} 
 \| Q_n v_0 - m(\tau, \omega) (P_n v_0)\|_{D(A^\alpha)}.
 \ee
   By \eqref{d1} we have
   \be\label{d8}
  \xi (0)
  =y_0 -
  \int_0^\infty  e^{A  s} P_n 
 \left ( F( \xi  + z  +  v  )
 -F(v  + z  ) \right ) ds.
 \ee
 By \eqref{d1} and \eqref{d8} we get
 \be\label{d10}
 \xi (t)
 =e^{-At} \xi (0)
 +
  \int_0^t   e^{-A (t-s)} 
 \left ( F( \xi   + z  +  v  )
 -F(v  + z  ) \right ) ds.
 \ee
  Note that $v(t,0, \omega, g^\tau, v_0)$ is a solution of  \eqref{a50}
  with $r =0$,  which along with \eqref{d10} implies
  that $v^*(t) = \xi (t) + v(t, 0, \omega, g^\tau, v_0)$ satisfies
 $$
  v^*(t)
  = e^{-At} v^*(0) +
  \int^t_0 e^{-A(t-s)} ( F(v^*(s) + z(\theta_s \omega)) + g(s+\tau) ) ds
  \ \ \mbox{with} \ v^*(0) = \xi (0) + v_0.
$$
This shows that $v^*(t)$ is  a solution of \eqref{a50} with initial condition
$v^*(0)$.  By the uniqueness of solutions, we have
$v^*(t) = v(t, 0, \omega, g^\tau, v^*(0) )$.
Since $\xi (t) = v(t, 0, \omega, g^\tau, v^*(0) )
-v(t, 0, \omega, g^\tau, v_0 )$, by \eqref{d7} we get
for all $t\ge 0$,
\be\label{d20}
\|v(t, 0, \omega, g^\tau, v^*(0) )
-v(t, 0, \omega, g^\tau, v_0 )\|_{D(A^\alpha)}
\le    {\frac 1{1-\delta}} e^{-\mu t} 
 \| Q_n v_0 - m(\tau, \omega) (P_n v_0)\|_{D(A^\alpha)}.
 \ee
By \eqref{d8} and \eqref{d2} we  get
$
v^*(0) = 
\xi (0) + v_0
=x_0 + m(\tau, \omega)  (x_0)   \in \calm(\tau, \omega)$ where
$x_0 = 
 P_n v_0
 -\int_0^\infty e^{As} P_n
 (F(v^*+ z )-
 F(v+ z ) ) ds$, which along with \eqref{d20} completes   the proof.
 \end{proof}

As an immediate consequence of Lemmas \ref{lemb4} and \ref{leme1},
we obtain the existence of inertial manifolds for equation
\eqref{a20}.

 \begin{cor}
 \label{imv}
 Suppose \eqref{f1}, \eqref{g1} and \eqref{b1}-\eqref{b3} hold
 with $k\in (0, \frac 12)$. Then the cocycle $\Psi$ defined by \eqref{a40}
 for equation \eqref{a20} has an inertial manifold 
 $\calm =\{ \calm(\tau, \omega): \tau \in \R,
 \omega \in \Omega \}$ as given by
 \eqref{b4}  and \eqref{c17}.
 \end{cor}

Based on Corollary \ref{imv}, we are able to establish
the existence of inertial manifolds for the non-autonomous
stochastic equation \eqref{pde1}.

\begin{thm}
 \label{imu}
 Suppose \eqref{f1}, \eqref{g1} and \eqref{b1}-\eqref{b3} hold
 with $k\in (0, \frac 12)$. Then the cocycle $\Phi$ defined by \eqref{a41}
 for equation \eqref{pde1} has an inertial manifold 
 $\widetilde{\calm} =\{ \widetilde{\calm} (\tau, \omega): \tau \in \R,
 \omega \in \Omega \}$ which is given by,
  for each $\tau\in \R$
 and $\omega \in \Omega$,
 \be\label{imu_1}
 \widetilde{\calm} (\tau, \omega)
 = \{ x+ z( \omega) +  m(\tau, \omega) (x):
 x\in  P_n H \}
 =\calm (\tau, \omega) + z(\omega).
 \ee
 \end{thm}
 
 \begin{proof}
  Given  $\tau \in \R$,
  $\omega \in \Omega$ and
  $x \in  P_n H$, let $y= x+ P_n z( \omega)$. 
  Then   we have
  $$
  x+ z( \omega) +  m(\tau, \omega) (x)
  =  y + Q_n z( \omega) + m(\tau, \omega) 
  (y-P_n z( \omega)),
  $$
  which along with \eqref{imu_1} implies
   \be\label{imu_2}
 \widetilde{\calm} (\tau, \omega)
 = \{  y + Q_n z( \omega) + m(\tau, \omega) 
  (y-P_n z( \omega) ): 
 y \in  P_n H \}
 = \{ y +\widetilde{m} (\tau, \omega) (y): y \in P_n H \}
 \ee
 where 
 \be\label{imu_3}
 \widetilde{m} (\tau, \omega)
 (\cdot) =   
 Q_n z( \omega) + m(\tau, \omega) 
  ( \cdot -P_n z( \omega) )
  \ee
    is  a Lipschitz map  from
  $P_nH$ to $Q_n H$.  
  By \eqref{a42}, \eqref{imu_1}   and  the invariance of $\calm$ under $\Psi$ we get
  $$
  \Phi (t, \tau, \omega, \widetilde{\calm} (\tau, \omega))
  = \{ \Phi (t, \tau, \omega,
  x+ z(\omega) + m(\tau, \omega) (x): x  \in P_n H \}
  $$
  $$
  = \{ \Psi  (t, \tau, \omega,
  x+  m(\tau, \omega) (x)  + z(\theta_t \omega):  x  \in P_n H \}
  =
  \Psi (t, \tau, \omega, \calm (\tau, \omega) ) + z(\theta_t \omega)
  $$
 \be\label{imu_5}
  =
  \calm (\tau +t, \theta_t \omega)  + z(\theta_t \omega)
  =\widetilde{\calm} (\tau +t, \theta_t \omega) ,
  \ee
  where the last equality follows from \eqref{imu_1} again.
  Therefore, $\widetilde{\calm}$ is invariant under $\Phi$.
  Finally, we prove the attraction property of $\widetilde{\calm}$.
  Given $u_0 \in D(A^\alpha)$, let $v_0 = u_0 -z(\omega)$.
  By Lemma \ref{leme1}, there exists  
   a random variable $v_0^* (\tau, \omega) \in \calm
 (\tau, \omega)$ such that for all $t\ge 0$, 
 \be\label{imu_10}
 \|\Psi(t,\tau, \omega, v_0^*)
 -\Psi (t, \tau, \omega, v_0)\|_{D(A^\alpha)}
 \le
 {\frac 1{1-\delta}} e^{-\mu t}
 \| Q_n v_0 - m(\tau,\omega) (P_n v_0) \|_{D(A^\alpha)}.
\ee
 Let $u_0^* = v_0^* + z(\omega)$.
 Since $v_0^*   \in \calm
 (\tau, \omega)$, by \eqref{imu_1} we  find 
 $u_0^* \in \widetilde{\calm}(\tau, \omega)$.
 By \eqref{a42}  we get
$$
 \|\Phi(t,\tau, \omega, u_0^*)
 -\Phi (t, \tau, \omega, u_0)\|_{D(A^\alpha)}
=   \|\Psi(t,\tau, \omega, v_0^*)
 -\Psi (t, \tau, \omega, v_0)\|_{D(A^\alpha)}
$$
which along with \eqref{imu_10} yields
\be\label{imu_20}
 \|\Phi(t,\tau, \omega, u_0^*)
 -\Phi (t, \tau, \omega, u_0)\|_{D(A^\alpha)}
 \le
 {\frac 1{1-\delta}} e^{-\mu t}
  \| Q_n u_0    - \widetilde{m} (\tau,\omega) (P_n u_0)\|_{D(A^\alpha)}.
\ee
By \eqref{imu_2}, \eqref{imu_5} and \eqref{imu_20} we conclude
the proof. 
 \end{proof}
 
 We now establish relations between tempered  random
 attractors and inertial  manifolds.
 Recall that a tempered family 
  $\cala=\{ \cala(\tau, \omega): \tau \in \R, \omega 
 \in \Omega \}$ of nonempty compact subsets of   $D(A^\alpha)$
 is called a tempered random pullback attractor of $\Phi$
 if $\cala$ is measurable, invariant, and pullback attracts
 all tempered  family of bounded subsets of $D(A^\alpha)$
 (see, e.g., \cite{wan1} and the references therein).
 Based on this notation, the random inertial manifold
 constructed in Theorem \ref{imu}  must contain
 tempered random attractors as subsets.
 
\begin{thm}
 \label{imatt}
 Suppose \eqref{f1}, \eqref{g1} and \eqref{b1}-\eqref{b3} hold
 with $k\in (0, \frac 12)$. If  $\Phi$ 
 has a tempered pullback random attractor
 $\cala=\{ \cala(\tau, \omega): \tau \in \R, \omega 
 \in \Omega \}$ in $D(A^\alpha)$,
  then we  have
 $\cala (\tau, \omega) \subseteq \widetilde{\calm}(\tau, \omega)$
 for all $\tau \in \R$  and $\omega \in \Omega$, where 
 $\widetilde{\calm}(\tau, \omega)$ is given by \eqref{imu_1}.
 \end{thm}
 
 \begin{proof}
 Given $u \in \cala(\tau, \omega)$ and $t_m \to \infty$, by the 
 invariance  of $\cala$, we find that for every $ m \in \N$, there
 exists $u_m \in \cala (\tau -t_m,  \theta_{-t_m} \omega)$ such that
 \be\label{imatt_1}
 u = \Phi (t_m, \tau - t_m, \theta_{-t_m} \omega, u_m ).
 \ee
 By \eqref{imu_20},  there exists
 $u^*_m \in \widetilde{\calm} (\tau- t_{m},
 \theta_{-t_m} \omega )$
 such that
 \be\label{imatt_2}
 \| \Phi (t_m, \tau - t_m, \theta_{-t_m} \omega, u^*_m)
 -
 \Phi (t_m, \tau - t_m, \theta_{-t_m} \omega, u_m)\|_{D(A^\alpha)}
 $$
 $$
 \le 
 {\frac 1{1-\delta}}
 e^{-\mu t_m}
 \| Q_n u_m -\widetilde{m}(\tau -t_m, \theta_{-t_m} \omega) (P_n u_m)
 \|_{D(A^\alpha)}.
 \ee
 Since  $u^*_m \in \widetilde{\calm} (\tau- t_{m},
 \theta_{-t_m} \omega )$, by the invariance of $\widetilde{\calm}$ we 
 have $ \Phi (t_m, \tau - t_m, \theta_{-t_m} \omega, u^*_m)$
 $ \in \widetilde{\calm} (\tau, \omega)$.
  By \eqref{imu_2}
  we may write
  $\Phi (t_m, \tau - t_m, \theta_{-t_m} \omega, u^*_m)
  =x_m + \widetilde{m} (\tau, \omega) (x_m)$
  for some $x_m \in P_n H$.
  This along with   
  \eqref{imatt_1}-\eqref{imatt_2}
   implies
  \be\label{imatt_3}
  \|( x_m-P_n u) + ( \widetilde{m} (\tau, \omega) (x_m) -Q_n u) \|_{D(A^\alpha)}
  \le 
 {\frac 1{1-\delta}}
 e^{-\mu t_m}
 \| Q_n u_m -\widetilde{m}(\tau -t_m, \theta_{-t_m} \omega) (P_n u_m)
 \|_{D(A^\alpha)}.
 \ee
  By   \eqref{imu_3}, \eqref{c16}
   and $u_m \in \cala (\tau -t_m,  \theta_{-t_m} \omega)$ we get
 $$
 \| Q_n u_m -\widetilde{m}(\tau -t_m, \theta_{-t_m} \omega) (P_n u_m)
 \|_{D(A^\alpha)}
 \le
   \| u_m \| _{D(A^\alpha)}
 + \|  \widetilde{m}(\tau -t_m, \theta_{-t_m} \omega) (P_n u_m)
 \|_{D(A^\alpha)}
  $$
  $$
 \le 
  \| u_m \| _{D(A^\alpha)}
  + \| A^\alpha z(\theta_{-t_m} \omega)\|
 + \|  {m}(\tau -t_m, \theta_{-t_m} \omega) (P_n (u_m - z(\theta_{-t_m} \omega)))
 \|_{D(A^\alpha)}
 $$
  $$
 \le 
  {\frac {2-k}{1-k}} \| u_m \| _{D(A^\alpha)}
   + {\frac {2-k}{1-k}} \| A^\alpha z(\theta_{-t_m} \omega)\|
 $$
 $$
 + {\frac {k}{1-k}} \sup_{r\le 0} e^{\mu r} \|z(\theta_{r-t_m} \omega)\|_{D(A^\alpha)}
 + {\frac 1{1-k}}
 \int_{-\infty}^0
 e^{\lambda_1 s}\|g(s+\tau -t_m)\|_{D(A^\alpha)} ds
 $$
  $$
 \le 
    {\frac {2-k}{1-k}} \| \cala(\tau - t_m, \theta_{-t_m} \omega ) \| _{D(A^\alpha)}
  +   {\frac {2-k}{1-k}}  \|  z(\theta_{-t_m} \omega)\|_{D(A^\alpha)}
 $$
 \be\label{imatt_5}
 + {\frac {k}{1-k}} e^{\mu t_m}  \sup_{s\le -t_m} e^{\mu s}
  \|z(\theta_{s} \omega)\|_{D(A^\alpha)}
 + { \frac {e^{\mu t_m}}{1-k}}  \int_{-\infty}^{-t_m}
 e^{\lambda_1 r}\|g(r+\tau )\|_{D(A^\alpha)} ds.
\ee
 Taking the limit of \eqref{imatt_3} as $m \to \infty$,
 by \eqref{imatt_5}, \eqref{g2}
 and the temperedness of  $\cala$  and $z$, we get
 $$
 \lim_{m \to \infty}
  \|( x_m-P_n u) + ( \widetilde{m} (\tau, \omega) (x_m) -Q_n u) \|_{D(A^\alpha)}
 =0$$
 which implies
 $x_m \to P_n u$  and
 $\widetilde{m} (\tau, \omega) (x_m) \to Q_n u$.
 By the continuity of $\widetilde{m} (\tau, \omega)$ we obtain
 $Q_n u = \widetilde{m} (\tau, \omega) (P_n u)$
 and hence $u \in \widetilde{\calm} (\tau, \omega)$
 for any $u \in \cala (\tau, \omega)$ as desired.
  \end{proof}

\section{Periodicity and Almost Periodicity of Inertial Manifolds}
\label{pap}
\setcounter{equation}{0}

In this section, we assume the external function $g$
in \eqref{pde1} is time periodic or almost periodic, and
establish the pathwise periodicity or almost  periodicity
of the random inertial manifolds constructed in the
previous section.
If    $g: \R \to D(A^\alpha)$ is  
almost periodic, then $g$ must be bounded; that is,
 $\sup\limits_{t\in \R} \| g(t)\|_{D(A^\alpha)}
<\infty$. Therefore, in this case,  condition
\eqref{g1} is trivially fulfilled.

The main  results of  this section are given below.

\begin{thm}
 \label{ap1}
 Suppose \eqref{f1}  and \eqref{b1}-\eqref{b3} hold
 with $k\in (0, \frac 12)$.
 If $g: \R \to D(A^\alpha)$ is almost periodic, then
 $\Phi$ has an almost  periodic random inertial
 manifold 
 $\widetilde{\calm}(\tau, \omega)$  as given by \eqref{imu_1}.
 \end{thm}
 
 \begin{proof}
By the almost periodicity of $g$, 
for every $\varepsilon >0$, there
exists a positive number  $l= l(\varepsilon)$ such
that any interval of length $l$ contains a  point
$\tau_0$  such that 
\be\label{app1}
\| g(r + \tau_0) -g(r) \|_{D(A^\alpha)} \le 
{\frac 12} \varepsilon (1-k)\lambda_n,
 \ \ \mbox{for all} \ \ 
 r  \in \R.
 \ee
 Based on \eqref{app1} we will prove 
 \be\label{app3}
 \sup_{x \in P_n H}
 \|\widetilde{m}(\tau+\tau_0, \omega) (x)
 -\widetilde{m}(\tau, \omega) (x) \|_{D(A^\alpha)}
 \le \varepsilon,
 \ \ \mbox{for all} \ \ \tau \in \R ,
 \ee
 which will complete the proof.
 Let $\xi^*(x, \omega, r)$ be the unique
 fixed point of $\cali(\cdot, x, \omega, r)$ given by
 \eqref{c1}
 for every fixed $x\in P_nH$, $\omega\in \Omega$
 and $r \in \R$. Then by \eqref{c1}, \eqref{pn1}-\eqref{qn2},
 \eqref{f1} and \eqref{app1} we get
 for $t \le 0$,
 $$ e^{\mu t}
  \| \xi^*(x, \omega, \tau+\tau_0) (t)
 -
 \xi^*(x, \omega, \tau) (t) \|_{D(A^\alpha)}
 $$
 $$
 = e^{\mu t} \| \cali (
 \xi^*(x, \omega, \tau+\tau_0), x, \omega, \tau +\tau_0
 )
 -
  \cali (
 \xi^*(x, \omega, \tau), x, \omega, \tau 
 ) \|_{D(A^\alpha)}
 $$
 $$
 \le
 \int_t^0e^{\mu t}
 \| A^\alpha e^{-A(t-s)}
 P_n ( F(\xi^*(x, \omega, \tau+\tau_0) (s)
 + z(\theta_s \omega))
 -  F(\xi^*(x, \omega, \tau) (s)
 + z(\theta_s \omega))
 )\| ds
$$
$$
 +
 \int_{-\infty}^te^{\mu t}
 \| A^\alpha e^{-A(t-s)}
 Q_n ( F(\xi^*(x, \omega, \tau+\tau_0) (s)
 + z(\theta_s \omega))
 -  F(\xi^*(x, \omega, \tau) (s)
 + z(\theta_s \omega))
 )\| ds
$$
$$
+ e^{\mu t}\int_t^0 \| e^{-A (t-s)} P_n A^\alpha
( g(s+\tau +\tau_0) -g(s+\tau) ) \| ds
+ e^{\mu t}\int_{-\infty}^t \|  e^{-A (t-s)} Q_n A^\alpha
( g(s+\tau +\tau_0) -g(s+\tau) ) \| 
$$
$$
\le
L 
\| \xi^*(x, \omega, \tau+\tau_0) - 
\xi^*(x, \omega, \tau)\|_\cals
\left (
\lambda_n^\alpha \int_t^0 e^{(\mu-\lambda_n)(t-s)} ds
+
\int_{-\infty}^t
(\frac {\alpha^\alpha}{(t-s)^\alpha} +\lambda_{n+1}^\alpha )
e^{(\mu -\lambda_{n+1}) (t-s)} ds
\right )
$$
 $$
+ e^{\mu t}\int_t^0   e^{-\lambda_n  (t-s)}  
  \| g(s+\tau +\tau_0) -g(s+\tau)  \|_{D(A^\alpha)} ds
+ \int_{-\infty}^t   e^{- \lambda_{n+1}  (t-s)}  
\| g(s+\tau +\tau_0) -g(s+\tau)   \| _{D(A^\alpha)}
$$
$$
\le
L 
\left ( {\frac {\lambda_n^\alpha}{\mu-\lambda_n}}
+ {\frac {\lambda_{n+1}^\alpha +c_\alpha (\lambda_{n+1} -\mu)^\alpha}
{\lambda_{n+1} -\mu}}
\right )
\| \xi^*(x, \omega, \tau+\tau_0) - 
\xi^*(x, \omega, \tau)\|_\cals
$$
\be\label{app5}
+
{\frac 12} \varepsilon (1-k)\lambda_n
e^{\mu t}
 \int_t^0   e^{-\lambda_n  (t-s)}    ds
+ {\frac 12} \varepsilon (1-k)\lambda_n
 \int_{-\infty}^t   e^{- \lambda_{n+1}  (t-s)}   ds.
\ee
Since $t \le 0$ and  $\mu \in (\lambda_n, \lambda_{n+1})$, the last
integral in \eqref{app5}
is bounded by $\varepsilon (1-k)$, which together with
\eqref{b1} and \eqref{app5} implies
\be\label{app7}
  \| \xi^*(x, \omega, \tau+\tau_0)  
 -
 \xi^*(x, \omega, \tau)   \|_\cals
 \le k  \| \xi^*(x, \omega, \tau+\tau_0)  
 -
 \xi^*(x, \omega, \tau)   \|_\cals
 +\varepsilon (1-k).
 \ee
 Thus we get
 $  
  \| \xi^*(x, \omega, \tau+\tau_0)  
 -
 \xi^*(x, \omega, \tau)   \|_\cals
 \le \varepsilon$ and hence
 $  
  \| \xi^*(x, \omega, \tau+\tau_0)  (0)
 -
 \xi^*(x, \omega, \tau)  (0)  \|_{D(A^\alpha)}
 \le \varepsilon$.
 Since
 $\xi^*(x, \omega,  r)  (0) = x+m(r, \omega) x$
 for all $r\in \R$,  we
 obtain
 $\| m(\tau+\tau_0, \omega)(x)-m(\tau, \omega) (x) \|_{D(A^\alpha)}
 \le \varepsilon$ for all $x\in P_n H$, which along with \eqref{imu_3}
yields \eqref{app3},  and thus completes  the proof.
 \end{proof}
 
 Finally,  we present the pathwise periodicity of random inertial
 manifolds.
 
\begin{thm}
 \label{per1}
 Suppose \eqref{f1}  and \eqref{b1}-\eqref{b3} hold
 with $k\in (0, \frac 12)$.
 If $g: \R \to D(A^\alpha)$ is periodic with period $T>0$, then
 $\Phi$ has a $T$-periodic random inertial
 manifold 
 $\widetilde{\calm}(\tau, \omega)$  as given by \eqref{imu_1}.
 \end{thm}
 
 \begin{proof}
 Following the arguments of  \eqref{app7}, we obtain in the 
 present case that
 $$  
  \| \xi^*(x, \omega, \tau+T)  
 -
 \xi^*(x, \omega, \tau)   \|_\cals
 \le k  \| \xi^*(x, \omega, \tau+ T)  
 -
 \xi^*(x, \omega, \tau)   \|_\cals.
 $$
 Since $k\in (0,1)$ we get
  $ \| \xi^*(x, \omega, \tau+T)  
 -
 \xi^*(x, \omega, \tau)   \|_\cals =0$,
 and hence
 $  \xi^*(x, \omega, \tau+T)  (0)
 =
 \xi^*(x, \omega, \tau) (0)   $.
 As a consequence, we get
 $m(\tau+T, \omega) =m(\tau, \omega)$
 and thus
 $\widetilde{m}
 (\tau+T, \omega) =
 \widetilde{m} (\tau, \omega)$.
  \end{proof}

\end{document}